\documentclass[a4paper,12pt]{article}

\usepackage[left=1.4cm,right=1.4cm,top=2cm,bottom=2cm,bindingoffset=0cm]{geometry}
\usepackage[cp1251]{inputenc}
\usepackage{amsmath}
\usepackage{amsthm}
\usepackage{amssymb}
\usepackage{tikz}
\usepackage{graphicx}
\usepackage{wrapfig}

\usetikzlibrary{patterns}

\newtheorem{theorem}{Theorem}
\newtheorem{lemma}{Lemma}
\def\({\left(}
\def\){\right)}

\begin{document}

\author{{\bf  E.\,V.~Shchepin }}
\thanks{ Steklov Math. Institute scepin@mi-ras.ru}
\thanks{This work was supported from a grant to the Steklov International Mathematical Center in the
framework of the national project "Science" of the Russian Federation..}

\title{About the Serpinsky-Knopp curve.}

\date{}

\begin{abstract}
The Serpinsky-Knopp curve is characterized as the only curve (up to isometry) that maps a unit segment onto a triangle of a unit area,
so for any pair of points in the segment, the square of the distance between their images does not exceed four times the distance between them.
\end{abstract}

\paragraph{keywords} plane Peano curves,  square-to-linear ratio, locality
Sierpinski- Knopp curve, space-filling curves
\maketitle

The known Serpinsky-Knopp curve \cite{Sagan} maps the segment to an isosceles rectangular
 triangle in such a way that one half of the segment is mapped to half of the triangle, and the other half --- the other half.
In this case, any of the" halves " of the  curve is similar to the entire curve.
Therefore, each of the halves, in turn, is divided in half, so that the first quarter of the segment is displayed on the first quarter of the triangle, the second --- on the second, etc.

\begin{wrapfigure}[7]{r}{110pt}

\begin{tikzpicture}[scale=1,>=latex]
\filldraw[red](2,-0.5) node[scale=1,black]{};
\draw[line width=0.5pt](0,0)--(2,2); 
 \draw[line width=0.5pt](2,2)--(4,0);
 \draw[line width=0.5pt](4,0)--(0,0);

 \draw[line width=0.5pt](2,2)--(2,0); 
 \draw[line width=0.5pt](1,1)--(3,1); 
 \draw[line width=0.5pt](3,1)--(3,0);
 \draw[line width=0.5pt](1,1)--(1,0);


 \draw[line width=0.5pt](1.5,1.5)--(3,0);


\draw[line width=0.5pt](2.5,1.5)--(1,0);

 \draw[line width=0.5pt](1,0)--(0.5,0.5);  
 \draw[line width=0.5pt](3,0)--(3.5,0.5);   
 \draw[line width=0.5pt](2,0)--(1,1); 
 \draw[line width=0.5pt](2,0)--(3,1); 

 \filldraw[black](0.49,0.22) circle (0.5pt);

 \filldraw[](0.78,0.5) circle (0.5pt);

\filldraw[](1.22,0.5) circle (0.5pt);

\filldraw[](1.5,0.22) circle (0.5pt);

\filldraw[](1.78,0.5) circle (0.5pt);

\filldraw[](2.22,0.5) circle (0.5pt);

\filldraw[](2.5,0.22) circle (0.5pt);

\filldraw[](2.78,0.5) circle (0.5pt);

\filldraw[](3.5,0.22) circle (0.5pt);

\filldraw[](3.22,0.5) circle (0.5pt);

\filldraw[](1.5,0.75) circle (0.5pt);

\filldraw[](2.5,0.75) circle (0.5pt);

\filldraw[](1.5,1.22) circle (0.5pt);

\filldraw[](2.5,1.22) circle (0.5pt);

\filldraw[](2.22,1.5) circle (0.5pt);

\filldraw[](1.78,1.5) circle (0.5pt);

 \draw[line width=0.5pt](0.49,0.22)--(0.78,0.5);
 \draw[line width=0.5pt](0.78,0.5)--(1.22,0.5);
 \draw[line width=0.5pt](1.22,0.5)--(1.5,0.22);
 \draw[line width=0.5pt](1.5,0.22)--(1.78,0.5);
 \draw[line width=0.5pt](1.78,0.5)--(1.5,0.75);
 \draw[line width=0.5pt](1.5,0.75)--(1.5,1.22);
 \draw[line width=0.5pt](1.5,1.22)--(1.78,1.5);
 \draw[line width=0.5pt](1.78,1.5)--(2.22,1.5);
 \draw[line width=0.5pt](2.22,1.5)--(2.5,1.22);
 \draw[line width=0.5pt](2.5,1.22)--(2.5,0.75);
 \draw[line width=0.5pt](2.5,0.75)--(2.22,0.5);
 \draw[line width=0.5pt](2.22,0.5)--(2.5,0.22);
 \draw[line width=0.5pt](2.5,0.22)--(2.78,0.5);
 \draw[line width=0.5pt](2.78,0.5)--(3.22,0.5);
 \draw[line width=0.5pt,->](3.22,0.5)--(3.5,0.22);

\end{tikzpicture}
\end{wrapfigure}

The halving process continues indefinitely, resulting in a continuous mapping of the unit
segment $s\colon [0,1]\to \Delta$ onto a triangle, such that the image of any interval of type
$[\frac{k-1}{2^n},\frac{k}{2^n}]$, where $k\le 2^n$ are natural numbers, is a triangle,
 similar to the original with a similarity coefficient of $2^{n/2}$.
The segment $[\frac{k}{2^n}, \frac{k+1}{2^n}]$, following \cite{Sch2004}, is called \emph{fractal period} of the curve $s(t)$, and the restriction of the curve to a fractal period --- \emph{by the fraction of the order $n$} of this curve.
Different fractions of the same order are isometric with each other.
In figure 1, the broken line shows the order of passage of the fourth-order fractions of the Serpinsky-Knopp curve.

In the future, we will treat the variable $t$ as time.
The square-linear ratio (SLR) of a pair of points $p (t_1), p (t_2)$ of a plane curve is called, according to \cite{Sch2004}, the ratio of the square of the Euclidean distance between their images to the time interval between them:
\begin{equation}\label{slr}
  \frac{|p(t_2)-p(t_1)|^2}{t_2-t_1}
\end{equation}

The similarity of fractions preserves the square-linear ratio.
The exact upper bound of a SLR  pair of curve points is called the \emph{locality} of the curve.
Among all currently known curves that map a unit segment to a flat set of unit area, the Serpinsky-Knopp curve has the smallest locality(see\cite{Hav2}).
The fact that this curve has locality 4 is proved in the article \cite{Nie}.
We will give a simpler proof here.
\begin{lemma}\label{max}
The maximum SLR of the $s(t)$ curve is reached.
\end{lemma}
\begin{proof}
 Similar to the proof of theorem 1 of \cite{Sch2015}.
\end{proof}

\begin{lemma}\label{ugol} If the angle $\angle p(t_1)p(t_2)p(t_3)$ is not obtuse and $t_1<t_2<t_3$, then SLR pair $p(t_1) p(t_3)$ is not less than the maximum of the SLR pairs $p(t_1) p(t_2)$ and $p(t_2) p(t_3)$. And if the angle is sharp, it is less.
\end{lemma}
\begin{proof}
Indeed, the condition for an angle is equivalent to an inequality
\begin{equation*}
  |p(t_1)-p(t_2)|^2+|p(t_2)-p(t_3)|^2\le |p(t_3)-p(t_1)|^2,
\end{equation*}
\end{proof}

\begin{proof}[Proof that the locality of $s(t)$ is equal to four]
Consider a pair of points $s(t)$ $s (t')$ of the  Sierpinski-Knopp curve
 with a maximum SLR. Let the points $A$ and $C$ be the ends of the fraction
the largest order containing this pair. Then $s(t)$ and $s(t')$ lie in different
the halves of this faction. The middle of the fraction is $B$. According to
the lemme \ref{ugol} angle $\angle s(t)Bs (t')$ must be straight, and
The SLR of both pairs $s(t) B$ and $Bs(t')$ is also maximum. Thus
$s(t)$ lies on the side of the fraction, and $B$ is its vertex. Consider
the largest-order fraction containing the pair $s(t), B$. In that case
$s(t)$ lies on the first half of the hypotenuse of this fraction. Since
the rectangular vertex $B'$ of this fraction passes between $s(t)$, then
the angle $\angle s(t), B'b$ must be straight. And this is only possible
if $s(t)$ is the vertex of the fraction.
 But in this case, the SLR of the $s(t)$ $s(t')$ pair is equal to four.
\end{proof}

The main result of the article is as follows:
\begin{theorem}
  \label{knopp}
There is a unique, up to isometry, mapping of a unit segment onto a triangle of a unit area whose locality is four.
\end{theorem}

The proof of the theorem is preceded by several lemmas.
\begin{lemma}
\label{treug}
If a triangle of unit area has sides $a, b, c$ that satisfy the inequalities
  $c^2\le4\ge a^2+b^2$, then $a=b=\sqrt2$, $c=2$.
\end{lemma}
\begin{proof}
Consider a triangle
with the sides $a\le b\le c$, which satisfies the inequalities of the Lemma and has the maximum area.
Then, for him, the Lemma inequalities will obviously reverse the equality $c=2$, $a^2+b^2=4$.
 Moreover, this triangle will have equal sides $a=b$.
Indeed, if $a\ne b$, then the same height isosceles triangle with the base $c$ and the sides $a_1, b_1$ has the same area
and satisfies the strict inequality $a_1^2+b_1^2<4$, and therefore cannot have a maximum area.
Where do we get that the maximum area of a triangle satisfying the Lemma's inequalities is a rectangular isosceles,
and other triangles that satisfy the Lemma's inequalities have a smaller area.
\end{proof}

\begin{lemma}\label{median}

If the curve $p(t)\colon [a, b]\to R^2$ of locality 4 passes through all three vertices of the rectangular isosceles triangle $ABC$, whose area is equal to the length of the segment, then the beginning and end of the segment $[a,b]$ pass to the sharp-angled vertices $A$, $B$
a triangle, and the middle of the segment --- into its rectangular vertex $C$.
\end{lemma}
\begin{proof}
Since the area of the triangle $ABC$ is equal to $b-a$, its cathets are equal to $\sqrt{2(b-a)}$, and the hypotenuse is
equal to $2\sqrt{b-a}$. Let $t_A$,$t_B$, and $t_C$ denote the points of passing the vertices of the triangle.
Then by virtue of the condition on the locality of the curve
$|t_B-t_A|\ge \frac14|B-A|^2=b-a$, where $|t_A-t_B|=b-a$, which is possible only if the moments $t_A$ and $t_B$ coincide with the ends of the segment.
 Further, the inequalities $|t_A-t_C| \ge \frac14|C-A|^2=\frac12(b-a)=\frac14|B-A|^2\le |t_B-t_C|$ entail equality
distances from $t_C$ to both ends of the segment.
\end{proof}

\ \begin{wrapfigure}[8]{r}{90pt}

\begin{tikzpicture}[scale=0.9,>=latex]

\filldraw[red](1.7,-0.5) node[scale=0.7,black]{};

 \draw[line width=0.5pt,dash pattern=on 1pt off 3pt](0,0)--(0,3);
 \draw[line width=0.5pt,dash pattern=on 1pt off 3pt](0,3)--(3,3);
 \draw[line width=0.5pt,dash pattern=on 1pt off 3pt](3,3)--(3,0);
\draw[line width=0.5pt] (0,0)--(3,0);

\draw[line width=0.5pt,dash pattern=on 1pt off 3pt](1.5,1.5)--(3,3);
\draw[line width=0.5pt,dash pattern=on 1pt off 3pt](1.5,1.5)--(0,3);

 \draw[] (3,3) arc [start angle=0, end angle=-180,radius=1.5];
 \draw[] (0,0) arc [start angle=-90, end angle=90,radius=1.5];
 \draw[] (3,3) arc [start angle=-270, end angle=-90,radius=1.5];

  \draw[line width=0.5pt](0,0)--(1.5,1.5);
  \draw[line width=0.5pt](3,0)--(1.5,1.5);

 \filldraw[](0,3.2) node[scale=0.7]{$A^\prime$};
 \filldraw[](3,3.2) node[scale=0.7]{$C^\prime$};

 \filldraw[](0,-0.2) node[scale=0.7]{A};
 \filldraw[](3,-0.2) node[scale=0.7]{C};
 \filldraw[](1.6,1.3) node[scale=0.6]{B};

\end{tikzpicture}
\end{wrapfigure}

\begin{proof}[Proof of the theorem.]

Let a triangle with the sides $a\le b \le c$ of the unit area
is an image of the $p(t)$ curve with locality 4. Let's show that in this
in this case, the triangle is rectangular and isosceles. The condition for locality 4 gives inequalities $a^2+b^2\le 4\ge c^2$. So on
based on the \ref{treug} Lemma, we conclude that the
the triangle is isosceles and rectangular with a hypotenuse equal to 2.

It follows from the Lemma \ref{median} that the ends of the curve $p(t)$ coincide with sharp-angled vertices
triangle image. Consider the Serpinsky-Knopp curve $s(t)$  with the same image as $p(t)$,
 the beginning of which $s(0)$ coincides with $p(0)$. In this case, the same Lemma implies that
 $p(1)=s(1)$.  We will prove by induction on $n$ that
  \begin{equation}\label{2rat}
  p\(\frac{k}{2^n}\)=s\(\frac{k}{2^n}\)
  \end{equation}
 for all binary-rational points of a unit segment.
For $n=0$, our statement is proved above.
Assume that the equality \eqref{2rat} is true for all $k\le2^n$.
To prove the induction step, it is sufficient to make sure that
\begin{equation}\label{1rat}
  p\(\frac{2k+1}{2^{n+1}}\)=s\(\frac{2k+1}{2^{n+1}}\)\end{equation}
when $k<2^n$.
Consider the $ABC$ fraction of the $s(t)$ curve of order $n$, So that
$A=p(\frac{k}{2^n})$ and $C=p(\frac{k+1}{2^n}$ are its sharp-angled vertices, and
$B=s(\frac{2k+1}{2^{n+1}})$ --- rectangular.
It follows from the  Lemma \ref{median} that if
\begin{equation}\label{inclusion}
 B\in p\left[\frac{k}{2^n}, \frac{k+1}{2^n}\right],
\end{equation}
  then $B=p(\frac{2k+1}{2^{n+1}})$.
Therefore, for a complete proof of \eqref{2rat}, it is sufficient to prove

the inclusion \eqref{inclusion}. Consider all fractions of the$n $ -th order of the curve $s(t)$,
containing $B$. There are no more than four of them.
All of them are contained in a square with the center $B$ and the side $AC$. Let's denote the vertices of this square
$A'$ and $C ' $ (see figure 2). Others are located at a distance not less than the length of $AC$.
Therefore, $p$ - images of their fractal periods, coinciding at the ends with $s(t)$.
it is too far from $B$ (at least the length of $AC$) to enable it.
Therefore, $B$ can belong to
images of only the following four curve segments of type $p[\frac{i}{2^n},\frac{i+1}{2^n}]$:
these are segments of the curve with ends in $AC$, $AA'$, $A'C'$, $CC'$. According to the Lemma \ref{ugol}
these curve segments are contained in circles with diameters $AC$, $AA'$, $A'C'$, $CC'$
respectively. But the point $B$ is not internal to the union of three circles with diameters $AA'$, $A'C'$, $CC'$.
Therefore, there are points arbitrarily close to $B$
belonging to $p[\frac{k}{2^n}, \frac{k+1}{2^n}]$. Due to the closedness of the latter
we conclude that $B$ itself belongs to it.
Thus, it is established that the considered curve coincides with the Serpinsky-Knopp curve
in all binary-rational points. From the continuity of the curves, their complete coincidence follows.
\end{proof}

\end{document}